\documentclass[oneside,11pt]{amsart}

\usepackage[utf8]{inputenc}                

\usepackage[margin=1in]{geometry} 
\usepackage{mathtools}
\usepackage{amsthm,amssymb}  
\usepackage{xypic}      
\usepackage{graphicx}
\usepackage{epstopdf}
\usepackage{array}
\usepackage{xcolor}
\usepackage{enumitem}
\usepackage[hidelinks]{hyperref}
\usepackage{microtype}
\usepackage{booktabs}

\newtheorem{theorem}{Theorem}[section]
\newtheorem{proposition}[theorem]{Proposition}

\newtheorem*{claim*}{Claim}
\newtheorem{corollary}[theorem]{Corollary}

\theoremstyle{definition}
\newtheorem{definition}[theorem]{Definition}
\newtheorem{example}[theorem]{Example}

\newtheorem{remark}[theorem]{Remark}

\newcommand{\R}{\mathbf{R}} 
\newcommand{\Z}{\mathbb{Z}}

\DeclareMathOperator{\lk}{lk}
\newcommand{\lks}{\lk_\Sigma}

\DeclareMathOperator{\Conf}{Conf}
\DeclareMathOperator{\Fr}{Fr}

\newcommand{\symmhomology}[1]{\mathbf{S}(#1)}

\begin{document}

\title{Concordance of spatial graphs}

\author[E. Lappo]{Egor Lappo}
\address{Department of Mathematics, Stanford University}
\email{elappo@stanford.edu}

{\small\begin{abstract}
We define smooth notions of concordance and sliceness for spatial graphs. We prove that sliceness of a spatial graph is equivalent to a condition on a set of linking numbers together with sliceness of a link associated to the graph. This generalizes the result of Taniyama for $\theta$-curves.

\end{abstract}}

\maketitle

\section{Introduction}

Spatial graphs are commonly defined as embeddings of compact one-dimensional CW-complexes into a 3-sphere. Spatial graphs having special abstract topology, such as \emph{$\theta$-curves}, appear naturally in knot theory in the study of strongly invertible knots \cite{MR1102258}, DNA replication \cite{buck2017unknotting, odonnol2018}, and other topics. 

Several isotopy invariants of spatial graphs were proposed, based, among others, on combinatorics of planar diagrams \cite{bao2020alexander,kauffman1989invariants}, the Alexander module \cite{litherland1989alexander}, Heegaard Floer homology \cite{bao2014floer}, instanton homology \cite{krmrw}. Many of them are extensions of isotopy invariants for knots. Knot concordance is another equivalence relation on the set of all knots. Two knots $K$ and $K'$ are said to be \emph{smoothly concordant} if there is a smooth annular cobordism $f: K\times I \to S^3 \times I$ such that $f(x,0) = K$ and $f(x,1) = K'$. Slice knots are knots concordant to the unknot.
A natural question then is to extend the notion of concordance equivalence of knots to spatial graphs and propose concordance invariants. 

In this paper, we produce a way to reduce the question of whether a given spatial graph is slice to sliceness of a certain link that can be obtained from the graph. This work can be seen as an extension of Taniyama's work \cite{taniyama1993cobordism} from sliceness of $\theta$-curves to general spatial graphs. We note that concordance equivalence of knots is different depending on whether one is working with locally flat or smooth maps. We work in the smooth category, while Taniyama worked in the locally flat category. Analogues of our results could equally well be proven in the locally flat setting. 

We begin by defining a spatial graph as an injective map $f: \Gamma \to S^3$ from a finite one-dimensional CW-complex to $S^3$ such that $f$ is smooth on 1-cells (edges) of $\Gamma$ and in the neighborhood of each 0-cell (vertex), the image has a fixed form, with all edges lying in the same plane (Definition~\ref{def:spatial_graph}). The image of $f$ is denoted by $G = f(\Gamma)$. We proceed to define three equivalence relations on spatial graphs: \emph{isotopy}, which reflects the equivalence most widely used in the community, \emph{rigid isotopy}, in which a neighborhood of each vertex is preserved, and \emph{concordance}, in which two spatial graphs $G$ and $H$ are equivalent if there is an identity cobordism $G\times I$ in $S^3\times H$ from $G$ to $H$ (Definition~\ref{def:graph_equivalences}). 
This cobordism is required to satisfy certain smoothness conditions, being what we call a rigid embedding.
By analogy with knots, we define \emph{slice spatial graphs} as spatial graphs that are concordant to a \emph{planar} one, that is, to a graph embedded in $S^2\subset S^3$. 

We proceed to define a \emph{framing} $\Sigma$ of a spatial graph $G$ as an oriented surface that deformation retracts onto $G$. Equivalences of graphs naturally extend to the framed context (Definition~\ref{def:framed_graph_equivalences}).
Constituent knots, which are simple cycles in $G$, can then be ``pushed-off'' in a positive direction with respect to the framing surface $\Sigma$. Linking numbers of constituent knots obtained with the help of such push-offs are then shown to be invariant under framed concordances of spatial graphs (Proposition~\ref{prop:lk_no}) and such linking numbers vanish for framed slice spatial graphs. 

We then give an algebraic description of linking numbers of constituent knots that allows us to tell whether a given spatial graph can be framed slice even when no framing is given. More precisely, to each abstract graph topology $\Gamma$ we associate an abelian group that encodes all possible framings across all embeddings of $\Gamma$, modulo twists on edges. We call it the \emph{module of framings} $\Fr(\Gamma)$.
For each spatial graph $G$ having abstract topology $\Gamma$, we have an element $\mathbf K(G) \in \Fr(\Gamma)$. We show the following result:
\begin{theorem}\label{thm:lk_application}
    Let  $G$ and $H$ be spatial graphs.
    If there is a concordance between $G$ and $H$, then $\mathbf K(G) = \mathbf K(H)$.
\end{theorem} 
This gives us a concrete algebraic condition on the framings that can be calculated through linking numbers.

A collection of circles embedded in a framing surface $\Sigma$ of $G$ is called a \emph{link pattern}. For a special class of link patterns, called \emph{fundamental link patterns}, we prove that as long as the linking number condition vanishes, sliceness of a fundamental link pattern is equivalent to sliceness of the spatial graph $G$:
\begin{theorem}\label{thm:main}
    Let $G$ be a spatial graph with abstract topology $\Gamma$. Let $\{K_i\}$ be its constituent knots constructed from a maximal tree as in Section~\ref{subsec:linking_homology}, and $\Sigma$ an arbitrary framing of $G$. Then, $G$ being slice is equivalent to the combination of the following two conditions:
    \begin{enumerate}[label=(\roman*)]
        \item The images of push-offs of constituent knots are zero in the module of framings of $\Gamma$, $\mathbf K(G) = 0 \in \Fr(\Gamma)$;
        \item if (i) is true, then there is a framing $\Sigma_0$ such that each for push-off we have $[K_i^+] = 0 \in H_1(S^3- G)$. The second condition then is: a fundamental link pattern in $\Sigma_0$ is slice.
    \end{enumerate}
\end{theorem}
This generalizes the result of Taniyama \cite{taniyama1993cobordism} for $\theta$-curves. In that case, condition (i) is vacuous and we only have condition (ii). 

\subsection*{Acknowledgements}

The author would like to thank Ciprian Manolescu, for suggesting the problem, as well as for his immense support and patience during writing this paper.

\section{Definitions}\label{sec:def}

In this section we recall some topological preliminaries and introduce definitions for spatial graphs and their concordance.

\subsection{Spatial graphs}\label{subsec:spatial_graphs}

A \emph{graph} $\Gamma$ is a finite one-dimensional CW-complex. The 0-cells of $\Gamma$ are called \emph{vertices}, and 1-cells are called \emph{edges}. The number of edges adjacent to the vertex is called the order of the vertex. In this paper, all graphs are assumed to have no vertices of orders 1 or 2, except for the case of links, which we see as spatial graphs having connected components with a single edge and a single vertex of order two.

An \emph{orientation} of a graph $\Gamma$ is a choice of a ``direction'' for each edge of $\Gamma$, and a labelling of a graph is a map associating a unique label to each cell of $\Gamma$. In this paper, all graphs are assumed to be oriented and labeled.

A \emph{spatial graph} is usually taken to be a graph $\Gamma$ together with an injective map $f: \Gamma \hookrightarrow S^3$ having some additional properties. The simplest example would be to require $f$ to be continuous or of class $C^k$ on the \emph{edges} of $\Gamma$, as in \cite{gulliver2008total}. 
However, these conditions are insufficient for our purposes, since pathological phenomena similar to \emph{wild knots} can occcur in the neighborhood of a vertex. We give a definition of a spatial graph in the context of smooth topology.

\begin{definition}\label{def:spatial_graph}
    A \emph{spatial graph} $G = (\Gamma, f)$ is an abstract graph $\Gamma$ together with an injective map $f: \Gamma \hookrightarrow S^3$ such that 
    \begin{enumerate}[label=(\alph*)]
        \item  $f$ is a smooth embedding on the 1-cells (edges) of $\Gamma$, and
        \item for every vertex $v$ of $\Gamma$, there exists a neighborhood $U$ of $f(v)$ such that $U\cap f(\Gamma)$ is diffeomorphic to $(D^3, X_n)$ with $n \geq 3$, shown in Figure~\ref{fig:x_n_nbhd}.
    \end{enumerate}
    Here, $X_n$ is a cone on $n$ points spaced regularly on the unit circle in the plane.
\end{definition}

\begin{figure}
    \centering
    \includegraphics[width=0.3\textwidth]{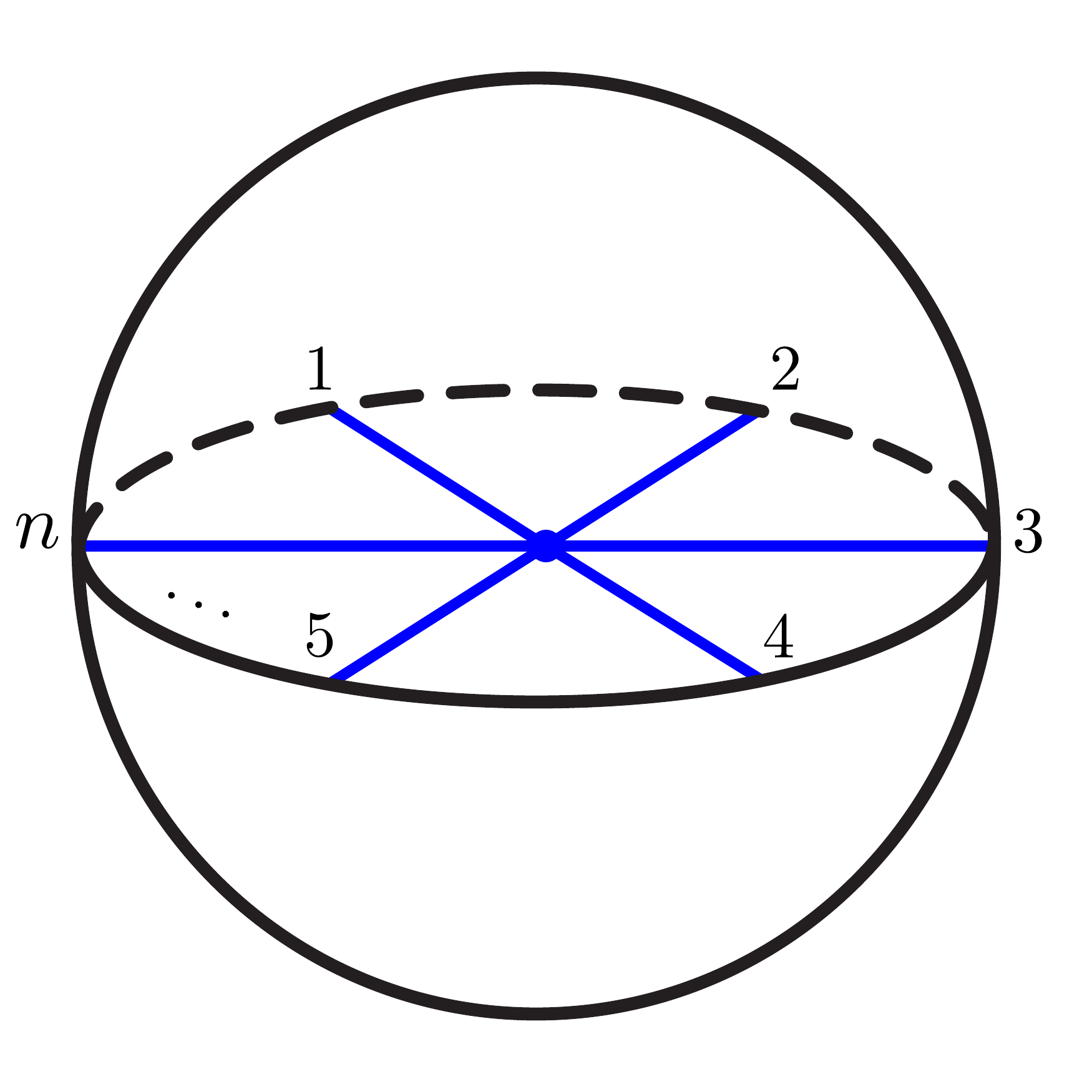}
    \caption{The pair $(D^3, X_n)$ from Definition~\ref{def:spatial_graph}, with $Y_n$ shown in blue. Notice that $X_n$ lies in a plane within $B^3$.}
    \label{fig:x_n_nbhd}
\end{figure}

\begin{remark}
In the literature, several additional conditions are often introduced on spatial graphs, such as: \emph{no sources or sinks}, so that each vertex is adjacent to both incoming and outgoing edges; \emph{transverse orientation}, so that at each vertex of $G$ there is a small embedded disk $D_\epsilon$ separating the incoming and outgoing edges; \emph{balanced coloring}, which is an asignment of a non-negative integer to each edge such that at each vertex the sum of integers on incoming edges is equal to the sum on outgoing ones, and so on (see \cite{bao2014floer,bao2020alexander,vance2020tau}). These conditions are usually required for algebraic invariants to work, and here we do not make use of them.
\end{remark}

The subtlety in Definition \ref{def:spatial_graph} is highlighted in the discussion of \emph{spatial graph equivalences}. While a multitude of equivalence relations on spatial graphs already exists in the literature \cite{taniyama1994cobordism}, here we provide new definitions that are suitable for the study of four-dimensional phenomena in the smooth context. First, we give some auxiliary definitions.

\begin{definition}\label{def:map_properties}
    Let $I = [0,1]$, $G = (\Gamma, g)$ and $H = (\Gamma, h)$ be spatial graphs, and $f$ be a map $f: \Gamma\times I \to S^3\times I$. We say that
    \begin{enumerate}[label=(\alph*)]
        \item $f$ is \emph{from $g$ to $h$} if there exists $\varepsilon > 0$ such that $f(x,t) = (g(x),t)$ for all $t \in [0, \varepsilon)$ and $f(x,t) = (h(x),t)$ for all $t \in (1-\varepsilon, 1]$. 
        \item $f$ is \emph{level-preserving} if for each $t \in I$ there is a map $f_t: \Gamma\to S^3$ such that $f(g,t) = (f_t(g), t)$. 
        \item $f$ is a \emph{rigid embedding} if every point of the image $f(\Gamma)$ has a neighborhood diffeomorphic to either $(D^4,D^2)$ or $(D^3\times (0,1), X_n\times (0,1))$.
    \end{enumerate}
\end{definition}

Then, the equivalences are defined as follows.

\begin{definition}\label{def:graph_equivalences}
    \begin{enumerate}[label=(\alph*)]
        \item Let $\Conf_n(S^2) = \big\{\{x_1,\ldots, x_n\} \subset S^2 \mid x_i \neq x_j\big\}$ be the configuration space of $n$ points on a sphere and $\gamma: I \to \Conf_n(S^2)$ be a smooth map. We define $C_\gamma$ to be the image of taking cones on $\gamma(t)$, i.e. 
        \[
            C_\gamma = \bigcup_{t, u \in [0,1]} \{t\} \times \{u\cdot \gamma(t)\} \subset I\times D^3.    
        \]

        Then, we say that $G$ and $H$ are \emph{isotopic} if there exists a level-preserving map $f: \Gamma\times I \rightarrow S^3\times I$ smooth on the 2-cells of $\Gamma\times I$ and such that for all $t\in (0,1)$, all $v\in V(\Gamma)$, there exists a neighborhood $U$ of $f(v,t)$ and a level-preserving diffeomorphism $U\to C_\gamma$ for some $\gamma(t)$ as above.

        \item Two spatial graphs $G$ and $H$ are \emph{rigidly isotopic} if there exists a level-preserving rigid embedding $f: \Gamma\times I \rightarrow S^3 \times I$ from $G$ to $H$.
        \item Two spatial graphs $G$ and $H$ are \emph{concordant} if there exists a rigid embedding $\Gamma\times I\rightarrow S^3 \times I$ from $G$ to $H$.
    \end{enumerate}
\end{definition}

Intuitively, Definion \ref{def:graph_equivalences}(a) allows for arbitrary movement of edges at a vertex during the isotopy, while restricting the behavior so that pathological phenomena like wild knots or infinite braiding do not occur. Definition~\ref{def:graph_equivalences}(b) restricts the isotopy further, requiring that a small neighborhood of each vertex is fixed during the isotopy. 

\subsection{Planar projections}\label{subsec:planar_proj}

To each spatial graph one can associate a planar projection consisting of a set of vertices, arcs, and crossings. Planar projections of spatial graphs can be related via a sequence of \emph{Reidemeister moves} \cite{kauffman1989invariants, mellor2018invariants}, as presented in Figure~\ref{fig:graph_reidemeister}. For the sake of visual clarity, in the figures the angles between the edges at a vertex are not drawn equal.

The relationship between planar projections and our definitions in Section~\ref{subsec:spatial_graphs} is contained in the following proposition.
\begin{proposition}
    \begin{enumerate}[label=(\alph*)]
        \item Two spatial graphs are isotopic if and only if they possess planar projections related to each other via a sequence of Reidemeister moves as in Figure~\ref{fig:graph_reidemeister}.
        \item Two spatial graphs are rigidly isotopic if and only if they possess planar projections related to each other via a sequence of Reidemeister moves (I)-(V), not involving the move (VI).
    \end{enumerate}
\end{proposition}
Statement (a) is widely known. Part (b) has a similar proof, which follows from an observation that moves (I)-(V) leave a small neighborhood of each vertex invariant, as required by Definition~\ref{def:graph_equivalences}(b) (see Theorem~2.1 in \cite{kauffman1989invariants} for the proof of part (a) for trivalent graphs, and Section~III of the same paper for the proof of (b) for $4$-valent graphs).

\begin{figure}
    \centering
    \includegraphics[width=0.8\textwidth]{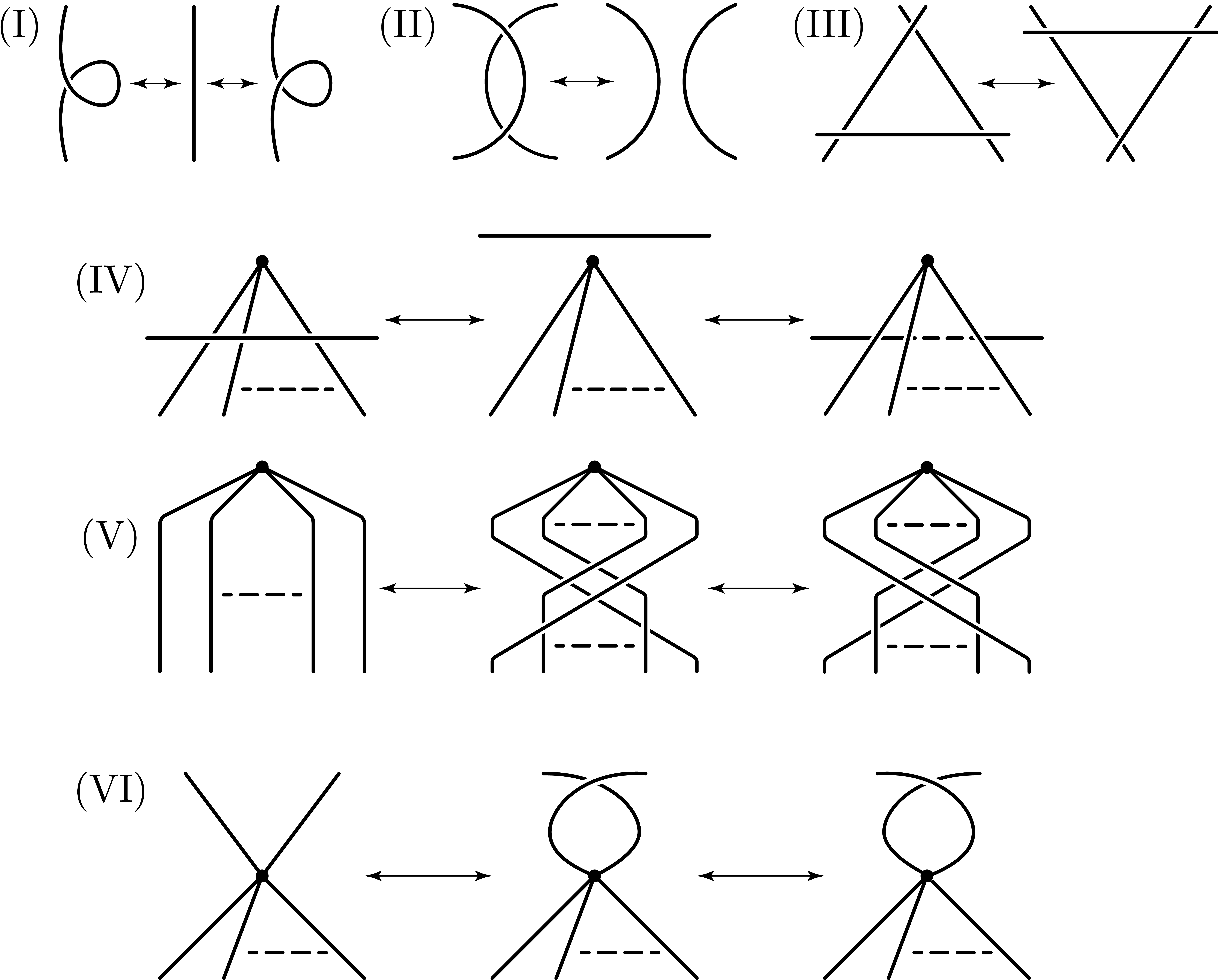}
    \caption{Reidemeister moves for spatial graphs. }
    \label{fig:graph_reidemeister}
\end{figure}

\subsection{Graph concordance}\label{subsec:conc}

The main focus of this paper is \emph{concordance} of spatial graphs, which we study in a smooth context. 
%Let $\cg$ be the set of equivalence classes of spatial graphs under the concordance relation. This set can be presented a disjoint union over the set of distinct abstract graph topologies: 
% \begin{equation}
%     \cg = \coprod_{\text{abstract graph $\Gamma$}} \cg(G).
% \end{equation}

% Given two spatial graphs $G$ and $H$, and a choice of edges $v \in G$, $w \in H$, one can define \emph{connected sum} $G\sharp H$ as follows: remove a small ball from a 3-sphere containing $G$ such that it meets $v$, and similarly for $H$. Then, identify the resulting boundary spheres such that their intersections with $G$ and $H$ match up and respect orientation of edges. Just as in the case of knots, it is easy to see that the connected sum operation descends on $\cg$. Moreover, the unit of this operation is the unknot, viewed as a spatial graph. Note that, because in general connected sum $G\sharp H$ can have more vertices than $G$, the connected sum operation in general does not have an inverse.

\begin{definition}\label{def:planar_slice}
    \begin{enumerate}[label=(\alph*)]
        \item A spatial graph with an embedding into a standard $S^2 \subset S^3$ is called \emph{planar}. Equivalently, a planar spatial graph is a spatial graph having a planar projection with no crossings.
        \item A spatial graph concordant to a planar spatial graph is called \emph{slice}.
    \end{enumerate}
\end{definition}

Note that every abstract planar graph has a unique isotopy class of embeddings into $S^3$ \cite{mason1969}. % Further, note that connected sum of planar graphs is also planar. it can be shown by defining graph genus $g(G)$ to be the minimal genus of a surface in which $G$ can be embedded without self crossings, then observing that $g(P) = 0$ for planar graphs $P$ and that $g(P\sharp Q) \leq g(P) + g(Q)$.

Immediately from Definition~\ref{def:planar_slice} we get the first obstruction to a spatial graph being slice: the case when the abstract graph $\Gamma$ does not have a planar embedding at all. By Kuratowski's theorem, this is true if and only if $\Gamma$ contains a subgraph $\Gamma' \subset \Gamma$ that is a subdivision of $K_5$ (the complete graph on five vertices) or $K_{3,3}$ (the complete bipartite graph on six vertices).

A simple cycle in a spatial graph can be seen as a piecewise-smooth closed curve in $S^3$. Recalling that evey piecewise smooth curve can be smoothed, we make the following definition.
\begin{definition}\label{def:constituent_knot}
    A knot obtained by smoothing a simple cycle in a spatial graph $G$ is called a \emph{constituent knot} of $G$. 
\end{definition}
Every constituent knot of a planar graph is an unknot. To make the presentation more straightforward, we assume that in a given (abstract) graph $\Gamma$, all of its constituent knots $K_i$ are oriented and their orientation is fixed. 

Using constituent knots it is possible to formulate the following condition on sliceness of spatial graphs.
\begin{proposition}\label{prof:const_knot_slice}
    Every constituent knot $K$ of a slice spatial graph $G$ is slice.
\end{proposition}

\begin{proof}
    Restricting the concordance of $G$ to a simple cycle $C\subset G$ gives a concordance $C\times I$ from $C$ to an unknotted planar spatial graph with $n$ edges and $n$ vertices. We will show that $K$ is slice by smoothing the concordance $C\times I$.

    The concordance $C\times I$ consists of $n$ ``sheets'' of the form $e\times I$ where $e$ is an edge of $C$, meeting at the ``seams'' $v\times I$ for $v$ a vertex of $v$. Clearly, it is enough to only consider a local picture around each seam. As $C\times I \to S^3\times I$ is a rigid embedding, we know that for each point of the sheet $e\times (0,1)$ we can find a neighborhood diffeomorphic (as pairs) to $(\R^4,\R^2)$. Using such diffeomorphisms locally and a partition of unity on a neighborhood of $e\times I$ it is possible to construct a nonvanishing vector field $\mathbf v$ on a neighborhood $U$ of $e\times (0,1)$ such that $\mathbf v$ is normal to $e\times (0,1)$ along it and $\mathbf v$ is tangent to each of the two sheets adjacent to $e\times (0,1)$.  

    Around each seam the pair of vector fields constructed as above, one for each sheet, defines a parametrization of a ``corner'': a map $P\times I \to S^3\times I$ where \[P = \{x\geq 0, 0\leq y \leq 1\}\cup \{0\leq x\leq 1, y\geq 0\}\subset \R^2.\] The preimage of a small neighborhood of a seam is $V = \{xy=0, x\geq0,y\geq0\} \subset P\times I$. In $P\times I$, it can be smoothed by taking instead of $V$ a curve $V' = \{xy=\varepsilon, x\geq0,y\geq0\}$. Mapping $V'$ back to $S^3\times I$ we obtain a smoothing of a seam. Performing this operation at each seam gives a smoothing of $C\times I$ to $K\times I$, that is, a concordance of $K$ to the unknot.
\end{proof}

\subsection{Framed concordance}\label{subsec:framed_conc}

The concordance invariants defined below require us to consider \emph{framed spatial graphs}. We define the framing of a spatial graph $G$ to be an \emph{oriented} surface  $\Sigma$ in which $G$ sits as a deformation retract, and denote a framed spatial graph by $(G,\Sigma)$.
% fn
\footnote{Note that some authors \cite{bao2020alexander, thurston2002algebra} do not require the framing surface to be orientable.}
We call two framings $\Sigma$ and $\Sigma'$ of $G$ \emph{equivalent} if there is an ambient isotopy of $S^3$ preserving $G$ and taking $\Sigma$ to $\Sigma'$.

Given a framed spatial graph $(G,\Sigma)$, there is a nonvanishing vector field on $G$ that always points out of the positive side of $\Sigma$. We call it a \emph{framing vector field}.
Then, given a planar projection of a spatial graph $G$, we can define the \emph{blackboard framing of $G$} to be a surface $\Sigma$ such that the vector field pointing orthogonally to the plane is the framing vector field for $\Sigma$. Note that thanks to our definition of a rigid spatial graph, all framing surfaces have a fixed structure near a vertex -- namely the neighborhood of the edges in the equatorial plane of the neighborhood on Figure~\ref{fig:x_n_nbhd}.

With these definitions, we observe that the equivalence relations from Definition~\ref{def:graph_equivalences} can be extended to the framed case.
\begin{definition}\label{def:framed_graph_equivalences}
    Let $\Gamma$ be a graph, $G = (\Gamma,g)$ and $H = (\Gamma, h)$ be spatial graphs, and let $\Sigma_G$, $\Sigma_H$ be framings of $G$ and $H$, respectively. We say that
    \begin{enumerate}[label=(\alph*)] 
        \item $(G,\Sigma_G)$ and $(H,\Sigma_H)$ are \emph{framed isotopic} if there is an isotopy map $f: G\times I \to S^3\times I$ that extends to an isotopy $\tilde f: \Sigma_G\times I \to S^3 \times I$ such that for all $t_0 \in I$ we have $\tilde f(x,t_0) = (f_{t_0}(x),t)$ for a framed spatial graph $f_{t_0}:\Sigma_G \to S^3$.
        \item $(G,\Sigma_G)$ and $(H,\Sigma_H)$ are \emph{framed concordant} if there is an (unframed) concordance map $f: \Gamma\times I \to S^3\times I$ that extends to a smooth embedding $\tilde f: \Sigma_G\times I \to S^3 \times I$.
        \item In particular, $(G, \Sigma)$ is \emph{framed slice} if it is framed concordant to a planar graph $(P,\Sigma_P)$ where $\Sigma_P$ is the blackboard framing.
    \end{enumerate}
\end{definition}

The relationship between framed and unframed concordance is summarised in the following proposition.
\begin{proposition}\label{prop:framing_extend}
    Given a framed spatial graph $(G,\Sigma_G)$ and an (unframed) concordance from $G$ to $H$, there exists a framing $\Sigma_H$ of $H$ and a framed concordance from $(G,\Sigma_G)$ to $(H,\Sigma_H)$.
\end{proposition}
\begin{proof}
    Considering $G\times\{0\}\subset S^3$, we define a set of nonvanishing vector fields $\{\mathbf v_0, \ldots, \mathbf v_n\}$, where $n$ is the number of edges of $G$, as follows. First, the vector field $\mathbf v_0$ is chosen on $G\times\{0\}$ such that it is normal to $\Sigma_G$. Then, $\mathbf v_1,\ldots, \mathbf v_n$ are defined on a neighborhood of each edge $e_i$ so that each $\mathbf v_i$ is tangent to both edges adjacent to $e_i$, as well as tangent to the surface $\Sigma_G$. Then, these $n+1$ vector fields can be extended to the whole concordance, to obtain nonvanishing vector fields $\{\mathbf{\tilde{v}}_0, \ldots, \mathbf{\tilde{v}}_n\}$ such that:
    \begin{enumerate}[label=(\alph*)]
        \item $\mathbf{\tilde{v}}_0$ is defined on $G\times I$, and given a neighborhood $(U, G\cap U)$ of $v\times t_0 \in V(G)\times I$ diffeomorphic to $W=(D^3\times (0,1), X_n\times (0,1))$, the image of $\mathbf{\tilde{v}}_0$ in $W$ is \emph{normal} to the plane containing $X_n\times t$ for each $t\in(0,1)$;
        \item for $i>0$, the vector field $\mathbf{\tilde{v}}_i$ defined on a neighborbood of $e_i\times (0,1)$ as in the proof of Proposition~\ref{prof:const_knot_slice}: $\mathbf{\tilde{v}}_i$ is normal to the sheet $e_i\times (0,1)$ and tangent to adjacent sheets. Additionally, the image of $\mathbf{\tilde{v}}_i$ under the diffeomorphism $U\to W$ from (a) is \emph{tangent} to the plane containing $X_n\times t$ for each $t\in(0,1)$.
    \end{enumerate}
    These conditions first define a fixed extension of the vector fields $\{\mathbf v_1, \ldots, \mathbf v_n\}$ to a neighborhood of $G\times \{0\} \cup V(G)\times I$. With these definitions, we observe that the equivalence relations. Finally, we find a framing $\Sigma_G\times I$ of $G\times I$ by finding integral surfaces of $\mathbf{\tilde{v}}_i$ for all $i>0$ that are normal to $\mathbf{\tilde{v}}_0$ for each $t\in I$. 
\end{proof}

\section{Linking numbers}\label{sec:lk} 

In this section we describe a way to associate a set of linking numbers to a framed spatial graph, and an invariant of framed concordance arising from it. 

To achieve this, first observe that for a framed spatial graph $(G,\Sigma)$, a framing vector field of $\Sigma$ allows us to define a \emph{push-off} of any subgraph of $G$ in the direction of the vector field. 
Then, for any constituent knot $K \subset (G,\Sigma)$, we can define $K^+$ to be such a push-off of $K$.
This leads to the following definition.
\begin{definition}
    Given a framed spatial graph $(G,\Sigma)$ and two constituent knots $K_1$, $K_2$, we define their linking number as $\lks(K_1,K_2) = \lk(K_1,K_2^+)$.
\end{definition}
It is easy to check that $\lks(K_1,K_2) = \lks(K_2,K_1)$.

\begin{proposition}\label{prop:lk_no}
    For all constituent knots $K_i,K_j \subset G$, the linking numbers $\lks(K_i,K_j)$ are invariant under concordances of $(G,\Sigma)$.
\end{proposition}
Before discussing the proof, let us state an important corollary.
\begin{corollary}\label{cor:lk_condition}
    Given a slice spatial graph $G$, there exists a framing $\Sigma$ of $G$ such that for all constituent knots $K_i \subset G$, we have $\lks(K_i,K_j) = 0$. 
\end{corollary}
\begin{proof}
    This follows from Proposition~\ref{prop:framing_extend} by extending the concordance from the planar graph with its blackboard framing to $G$.
\end{proof}
In Section~\ref{subsec:linking_homology} below, we will show that this corollary allows us to apply the criterion in Proposition~\ref{prop:lk_no} to spatial graphs $G$ without explicitly using framings.

\begin{proof}[Proof of Proposition~\ref{prop:lk_no}]
    The proof mirrors the proof of Theorem~3 in \cite{hosokawa1967concept}. Just as with links, using Morse theory one can view concordances of spatial graphs as sequences of diagram moves of the following types:
    \begin{enumerate}[label=(\alph*)]
        \item Introducing a disjoint unknot to the diagram;
        \item Contracting a disjoint unknot component;
        \item Introducing a band between either two points on the same edge of the graph, or between an edge and an unknot component.
    \end{enumerate}
    
    Introducing or removing unknots to the diagrams does not affect the linking numbers. Moreover, when a band is introduced to the diagram, the new crossings always appear in pairs with opposite signs, such that the total linking of constituent knots is not affected.
\end{proof}

\begin{remark}
    In \cite{yasry}, a notion of spatial graph concordance is introduced in which bands could be attached between points on different edges of a graph. This definition is not equivalent to ours, as it allows for concordances which are not embeddings of $\Gamma \times I$ into $S^3 \times I$.
\end{remark}

\subsection{Space of all framings of a spatial graph}\label{subsec:all_framings}

To proceed, we need to describe the space of all framings of a given spatial graph $G$. For the sake of presentation, let us restrict ourselves to a connected planar spatial graph $P$ with its planar embedding. 
Let $\Sigma_P$ be the blackboard framing of $P$. 
The surface $\Sigma_P$ consists of a disk for each vertex of $G$ and a band for each edge, all lying in the plane. 
Any other framing of $P$ can be isotoped to be represented by a disk in the plane for each vertex of $P$, and a band with some number of \emph{twists} on it for each edge of $P$. As with arcs of a knot, we distinguish between a \emph{full twist} and a \emph{half-twist}. Examples of framings for a planar $\theta$-curve are presented in Figure~\ref{fig:edge_twists}.

To state and prove the main claim of the section, we need the following definition.
\begin{definition}\label{def:edge_cut}
    Given a connected abstract graph $\Gamma$, an \emph{edge cut} $C \subset E(\Gamma)$ is a set of edges such that $\Gamma- C$ is disconnected, and for any $e\in C$, the graph $(\Gamma - C) \cup e$ is connected.
\end{definition}
With this, we obtain the following result connecting any two framings of $P$.
\begin{proposition}\label{prop:all_framings}
    Any framing $\Sigma'$ of a planar spatial graph $P$ can be obtained from the blackboard framing $\Sigma_P$ by repeatedly applying the following operations to it:
    \begin{enumerate}[label=(\alph*)]
        \item introducing a full twist on $\Sigma_P$ to an edge of $P$;
        \item introducing a half-twist of the same orientation (positive or negative) to each edge of an edge cut of $P$.
    \end{enumerate}
\end{proposition}

\begin{proof}
    First of all, applying each of these operations yields an \emph{orientable surface}: this is clear for operation (a), and for operation (b) this follows from the fact that after applying it, every cycle of $P$ has an \emph{even} number of half-twists on it (if a cycle passes an edge in an edge cut, then it has to pass through another edge in the same cut, always picking up half twists in pairs.)

    Then, we need to show that given an arbitrary framing $\Sigma'$ of $P$, we can turn it into $\Sigma_P$ by applying operations (a) and (b). 
    First, we reduce to the case of $\Sigma'$ having either no twists or a single positive half-twist on each edge by applying operation (a) as necessary.
    Let $C \in E(P)$ be the set of edges of $P$ on which $\Sigma'$ has a half-twist. We finish the proof by showing that $C$ is a union of edge cuts.

    First, $P-C$ is disconnected: if it were connected, then choose $e \in C$ and complete it to a cycle with edges in $P-C$. This cycle has only one half twist and therefore $\Sigma'$ contains a M\"obius band, which is a contradiction.
    Therefore, $C$ has to contain an edge cut $C_1$. Consider the set $C' = C-C_1$. If $C'$ is empty, we are done, and if not, $P-C'$ is again disconnected and therefore contains another edge cut, so we can repeat the argument. As $C$ has a finite number of elements, we eventually find a presentation $C = C_1\cup \cdots\cup C_k$ of $C$ as a union of edge cuts.
\end{proof}

Proposition~\ref{prop:all_framings} is true for all spatial graphs, and the planarity was only used to obtain a distinguished (blackboard) framing for the ease of presentation.

\begin{figure}
    \centering
    \includegraphics[width=0.6\textwidth]{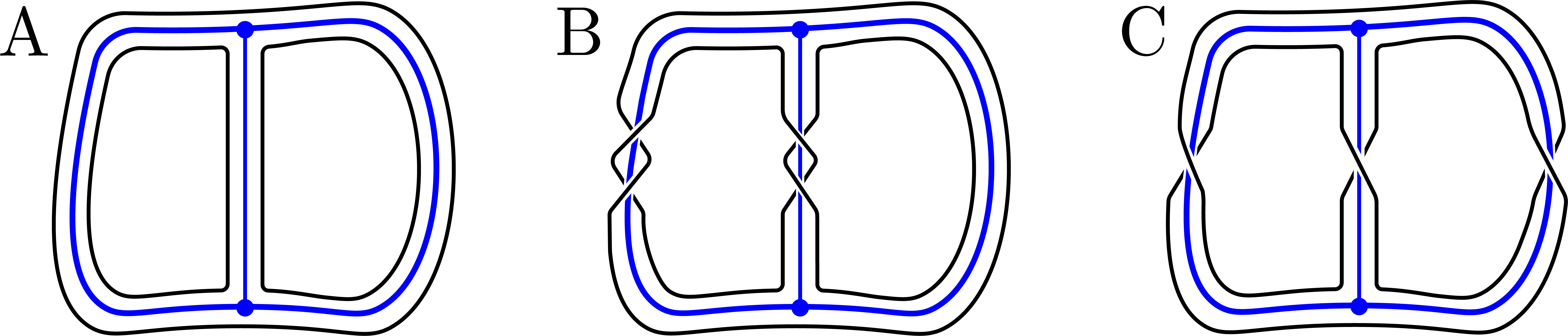}
    \caption{(A) A planar graph $P$, shown in purple, with a blackboard framing $\Sigma_0$ shown in black. (B) The same planar graph $P$ with a framing obtained from $\Sigma_0$ by introducing some twists on edges. (C) The same graph $P$ with a framing obtained from $\Sigma_0$ by applying a half-twist to each edge of an edge cut of $P$.} 
    \label{fig:edge_twists}
\end{figure}

\subsection{A homological perspective}\label{subsec:linking_homology}

Let us examine Corollary~\ref{cor:lk_condition} further. We will use Proposition~\ref{prop:all_framings} to suggest the way it can be applied. Let $ G = (\Gamma,f)$ be a spatial graph with $k$ edges and $b_1(G) = n$, and let $P = (\Gamma, p)$ be a planar embedding of $\Gamma$. 

For a fixed maximal tree $T\subset G$ we let $\{K_1\, \ldots, K_n\}$ be a set of constituent knots of $G$, defined as follows: for each edge $e_i\in E(G)-E(T)$, let $K_i$ be $e_i \cup T'$ where $T'\subset T$ is the unique path graph connecting the endpoints of $e_i$ through $T$. Then, by Alexander duality, $H_1(G)\cong H_1(S^3 - G) \cong \Z^n$, with the basis given by the meridians around each $e_i$.

To any framing $\Sigma$ of $G$ we can associate push-offs $K_i^+$ of constituent knots, each of which represents a homology class in $H_1(S^3 - G)$. These homology classes are determined by linking numbers between constituent knots. 

For a planar graph $P$ with a blackboard framing all linking numbers are zero. Then, Proposition~\ref{prop:lk_no} implies that the homology classes of push-offs of constituent knots are preserved under (framed) concordance, and Corollary~\ref{cor:lk_condition} says that if $G$ is slice (concordant to $P$), there is a framing $\Sigma$ of $G$ with $[K_i^+] = 0 \in H_1(S^3-G)$ for all $i$. 

Now, consider a spatial graph $G$ and an arbitrary framing $\Sigma$ of $G$. Applying a full twist to an edge $e_j$ of $G$ changes the framing, and therefore changes the homology classes of constituent knots $[K_i^+]$ by an element $t_j^{(i)} \in H_1(S^3 - G)$. In particular, $t_j^{(i)}$ is a meridian of an edge $e_j$ if $e_j \in K_i$, and zero otherwise. Similarly, applying a half twist to an edge cut $C\subset E(G)$ changes the homology of constituent knots $[K_i^+]$ by an element $c_j^{(i)} \in H_1(S^3 - G)$. We again have $c_j^{(i)} = 0$ if $K_i \cap C = \emptyset$. 

The information about all of the constituent knots can be put together by considering an element
\begin{equation}
    ([K_1^+],\ldots, [K_n^+]) \in H_1(S^3 - G)^n \cong \Z^{n^2}
\end{equation}
of the direct sum of first homology groups of the complement. Furthermore, we can view elements of $\Z^{n^2}$ as $n$-by-$n$ matrices $\{a_{ij}\}$, where the element $a_{ij}$ is the linking number between $K_i^+$ and $K_j$. As the linking numbers are symmetric, $([K_1^+],\ldots, [K_n^+])$ is a symmetric matrix. To account for this, let $\symmhomology{G}$ be the quotient of $H_1(S^3 - G)^n\cong \Z^{n^2}$ by a subgroup of differences of the off-diagonal elements. That is, if the generators of the $j$th copy of $H(S^3-G)$ in $H_1(S^3 - G)^n$ are denoted by $x_i^{j}$, then 
\begin{equation}
    \symmhomology G = \frac{\langle x_1^{1}, x_2^{1},\ldots, x_{n-1}^{n},x_n^{n}\rangle}{\langle x_i^{j} - x_j^{i}\text{ for $1\leq i < j \leq n$} \rangle}.
\end{equation}
We define $\mathbf{\widetilde{K}} (G,\Sigma) \in S(G)$ to be the image of $([K_1^+],\ldots, [K_n^+])$ in the symmetric quotient.

Then, the discussion of the effect of edge twists above implies that, given two framings $\Sigma$ and $\Sigma'$ of $G$, we have
\begin{equation}\label{eq:diff_K}
    \mathbf{\widetilde{K}}(G,\Sigma) -  \mathbf{\widetilde{K}}(G,\Sigma') = \sum_{e_j \in E(G)} a_j T_j + \sum_{\text{an edge cut }S_j} b_j C_j
\end{equation}
with $a_j,b_j \in \Z$ and
\begin{align*}
    T_j & = \bigoplus_{\text{constituent knots } K_i} t_j^{(i)}, \\
    C_j & = \bigoplus_{\text{constituent knots } K_i} c_j^{(i)}.
\end{align*}
Using this, we define a \emph{module of framings} 
\[
    \Fr(\Gamma) = \frac{\symmhomology G}{\langle T_1,\ldots, T_k, C_1, \ldots, C_r\rangle},
\] 
where $k$ is the number of edges of $G$ and $r$ is the number of edge cuts. Note that we write $\Fr(\Gamma)$ because the module is independent of a particular embedding of $\Gamma$: an isomorphism $H_1(S^3 - G) \cong H_1(S^3 - H)$ for any two spatial graphs $G$, $H$ with the same abstract topology $\Gamma$ implies that $\symmhomology{G} = \symmhomology{H}$ and therefore $\Fr(\Gamma)$ depends only on the abstract topology of the spatial graph. Below we will also write $\symmhomology{\Gamma}$ to reflect this, however this group still needs to be computed using a concrete spatial embedding.

Finally, we let $\mathbf{K}(G)$ to be the image of $\widetilde{\mathbf K}(G,\Sigma)$ in the quotient $\Fr(\Gamma)$.
With this notation, the following consequence of Corollary~\ref{cor:lk_condition} is true.

{
    \renewcommand{\thetheorem}{\ref{thm:lk_application}}
    \begin{theorem} 
        Let $G$ and $H$ be spatial graphs.
        If there is a concordance between $G$ and $H$, then $\mathbf K(G) = \mathbf K(H)$.
    \end{theorem}
    \addtocounter{theorem}{-1}
}

\begin{proof}
    Choosing an arbitrary framing $\Sigma_G$ on $G$, we can obtain a framed concordance from $(G, \Sigma_G)$ to $(H, \Sigma_H)$ by applying Proposition~\ref{prop:framing_extend}. Then, Proposition~\ref{prop:lk_no} says that linking numbers between constituent knots are preserved. Therefore, $\widetilde{\mathbf K}(G, \Sigma_G) = \widetilde{\mathbf K}(H,\Sigma_H)$ and hence $\mathbf K(G) = \mathbf K(H)$.
\end{proof}

\begin{corollary}\label{cor:lk_application_slice}
    If a spatial graph $G$ is slice, then $\mathbf K(G) = 0$.
\end{corollary}
\begin{proof}
    By Theorem~\ref{thm:lk_application}, $\mathbf K(G) = \mathbf K(P)$ where $P$ is a planar spatial graph. By choosing $\Sigma_P$ to be the blackboard framing of $P$, we see that $\widetilde{\mathbf K}(P,\Sigma_P) = 0$ and hence $\mathbf K(P) = 0$.
\end{proof}

\begin{example}\label{ex:linking_ex1}
Consider the following examples. Let $\Theta$ be a $\theta$-curve with edge labels and orientations as in Figure~\ref{fig:theta_labels}. A planar embedding $P$ of $\Theta$ is shown, and we will use it to calculate $\Fr(\Theta)$. There are two consituent knots $K_1=v_1\cup v_0$, $K_2=v_2\cup v_0$, and four operations on framings: a full twist on each edge $v_i$ or a half twist on an edge cut $\{v_0,v_1,v_2\}$. Let $x$ be the meridian of $v_1$ and $y$ be the meridian of $v_2$ in a copy of $H_1(S^3-P)$ associated to $K_1$, and let $z$ and $w$ be meridians of $v_1$ and $v_2$ for a copy of $H_1(S^3-P)$ associated to $K_2$. Using this, we can compute $\Fr(\Theta)$:
\begin{align}
    \Fr(\Theta) & = \frac{\symmhomology \Theta}{\langle t_0,t_1,t_2,c_1\rangle} \nonumber\\
           & = \frac{\langle x,y,z,w\mid y-z\rangle}{\langle x+y+z+w, x, w, x +z + w\rangle} \nonumber\\
           & = 0. \label{eq:Fr_G_theta}
\end{align}
Therefore, for \emph{any} spatial $\theta$-curve $G$ we can find a framing such that all linking numbers $\lks(K_i,K_j)$ are zero. This fact is well known in the literature on concordance of $\theta$-curves \cite{taniyama1993cobordism}. 
\end{example}

\begin{example}\label{ex:linking_ex2}
For spatial graphs having a more complex topology, the invariant is non-zero. As a natural extension of $\theta$-curves, consider a family of spatial graphs $\Theta_n$, such that $\Theta_n$ is homeomorphic to an (unreduced) suspension on $n+1$ points. Then, spatial embeddings of $\Theta_1$ are knots, and spatial embeddings of $\Theta = \Theta_2$ are the usual $\theta$-curves discussed above in Example~\ref{ex:linking_ex1}. Using edge labels analogous to those in Figure~\ref{fig:theta_labels} and letting $x_j^{i}$ to be the meridian of $v_j$ associated to the constituent knot $K_i$, it is possible to calculate
\begin{align}
    \Fr(\Theta_n) & = \frac{\symmhomology{\Theta_n}}{\langle t_0,\ldots, t_n, c_1 \rangle} \nonumber \\
                  & = \frac{\langle x_1^{1}, x_2^{1},\ldots, x_1^{2},\ldots, x_n^{n} \mid x_i^{j} - x_j^{i} \text{ for $1\leq i < j \leq n$}\rangle}{\langle \sum_{i,j}x_i^{j}, x_1^{1},x_2^{2},\ldots, x_n^{n}, \sum_{i=1}^n\sum_{j=1}^i x_j^{i} \rangle}. \label{eq:Fr_G_Theta_n}
\end{align}
For $n>2$ there are more generators than relations in the abelian group $\Fr(\Theta_n)$, so it is nontrivial. By considering the Smith normal form of the matrix representing $\Fr(\Theta_n)$, we see that the module of framings for $\Theta_n$ is free, 
\begin{equation}
    \Fr(\Theta_n)\cong \Z^{\frac12 n(n-1) - 2}. 
\end{equation}
We will see an example of a spatial $\Theta_3$-graph in Example~\ref{ex:theta 3 slice} below.
\end{example}

\begin{figure}
    \centering
    \includegraphics[width=0.25\textwidth]{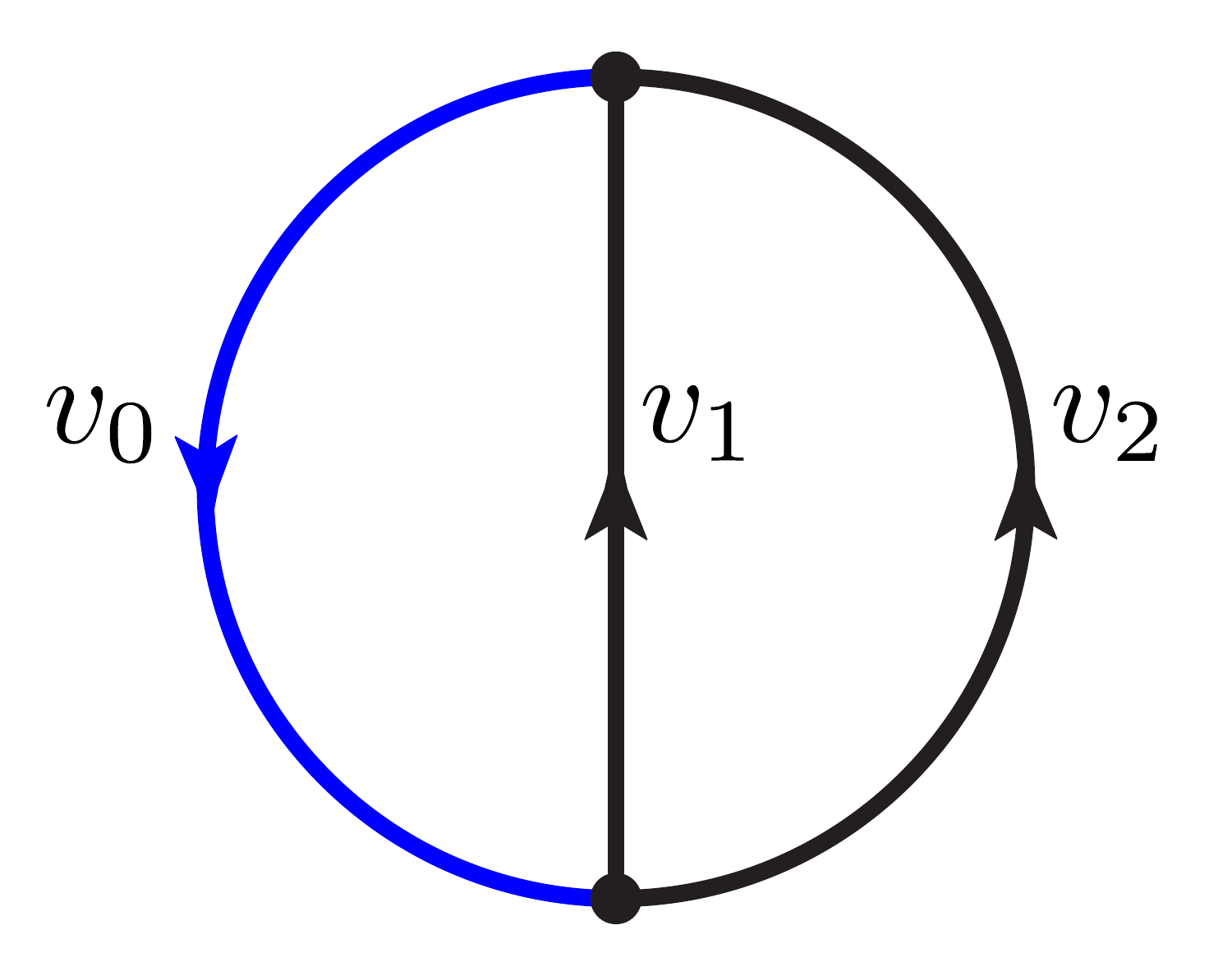}
    \caption{Edge labels and orientations for a planar $\theta$-curve. The maximal tree is a single edge shown in purple.} 
    \label{fig:theta_labels}
\end{figure}

\section{General link patterns}\label{sec:link_patt}

In this section we extend our consideration of linking numbers to a more general context of links ``in the neighborhood'' of a spatial graph $G$. In particular, we adapt Theorem~5 of \cite{taniyama1993cobordism} to general spatial graphs. 

\begin{definition}\label{def:link_pattern}
    Given a framed spatial graph $(G, \Sigma)$, a \emph{link pattern} is a collection of disjoint simple closed curves $\{L_1,\ldots,L_n\}$ in $\Sigma$.  
\end{definition}

Viewed as a submanifold of $S^3$, a link pattern defines a link. Recall that a link is called (strongly) slice if it bounds a collection of disjoint disks smoothly embedded in $B^4$. The link given by a link pattern can be used to give a condition on sliceness of $G$.

\begin{theorem}\label{thm:link_pattern}
    If $(G,\Sigma)$ is framed slice, then any link pattern $\{L_1,\ldots, L_k\}$ in $\Sigma$ is slice viewed as a link in $S^3$.
\end{theorem}
\begin{proof}
This follows quickly from observing that as $(G,\Sigma)$ is concordant to $(P,\Sigma_0)$ with $\Sigma_0\subset S^2 \subset S^3$, we can restrict the framed graph concordance to the link pattern to obtain a strong link concordance to a collection of unlinked circles in $S^2 \subset S^3$, which is clearly slice.
\end{proof}

For a specific link pattern it is possible to obtain a converse result. We begin with some necessary graph-theoretic definitions.
\begin{definition}\label{def:fnd_cycle}
    Given a graph $G$, a \emph{cycle basis} is a set of simple cycles forming a basis of $H_1(G,\Z_2)$. A cycle basis is called \emph{fundamental} if it is obtained from a maximal tree $T\subset G$ by associating to each edge in $e\in E(G-T)$ a cycle $e\cup P(e)$ where $P(e)\subset P$ is the unique path between enpoints of $e$ through $T$.
\end{definition}
We have already encountered this notion in Section~\ref{subsec:linking_homology}. It is known that a given cycle basis being fundamental is equivalent to each element of a basis possessing an edge not in any other elements \cite[Theorem 1]{syslo1979cycle}. We also observe that any fundamental link pattern must consist of $k = b_1(G)$ connected components. With this, we are able to make the definition of a fundamental link pattern for a framed spatial graph.

\begin{definition}\label{def:fnd_pattern}
    Let $\Sigma$ be a framing surface for some spatial graph $G$. A \emph{fundamental link pattern} $\{L_1,\ldots, L_k\}$ is defined to be a link pattern such that 
    \begin{enumerate}[label=(\alph*)]
        \item Each component $L_i$ maps to a simple cycle of $G$ under the deformation retraction of $\Sigma$ to $G$.
        \item Viewed as cycles, its components form a cycle basis with each element possessing an edge not belonging to other cycles.
    \end{enumerate}
\end{definition}

\begin{theorem}\label{thm:fnd_link_pattern}
    Let $(G,\Sigma)$ be a framed spatial graph. Then $(G,\Sigma)$ is framed slice if and only if a fundamental link pattern $\{L_1,\ldots, L_k\}$ in $\Sigma$ is slice.
\end{theorem}

%In the figures and in the proof below we consider a surface $\Sigma$ embedded in a plane when defining a link pattern. The discussion can then be transferred directly to any other embedding of $\Sigma$.

\begin{proof}
    The ``only if'' direction a restatement of Theorem~\ref{thm:link_pattern}. To show the opposite direction, we will construct a concordance for $G$ from a known concordance of the fundamental link pattern $(L_1\cup\cdots\cup L_k) \times I$. In this process, we will utilize the framing surface $\Sigma$, but will construct an unframed concordance for $G$. The outline is as follows:
    \begin{enumerate}[label=(\roman*)]
        \item For $0\leq t\le 1/4$ the concordance is simply the identity on $G$ and $\Sigma$.
        \item For $1/4\leq t\leq 1/2$ the surface $\Sigma$ is still kept constant, and $k-1$ bands are attached to edges of $G$ such that \begin{itemize}
            \item all bands and components of a graph lie inside $\Sigma$;
            \item after all $k-1$ bands are attached, the graph has $k-1$ components $S_1,\ldots, S_{k-1}$ homeomorphic to $S^1$ and one component homeomorphic to $G$, which we call the \emph{graph-component} $G'$;
            \item as curves within $\Sigma$, $S_i$ is isotopic (in $\Sigma$) to $L_i$ -- the $ith$ component of the fundamental link pattern, and the graph-component $G'$ can be isotoped to be located inside a small neighborhood of $L_k$ and is not contractible in $\Sigma$.
        \end{itemize}
        \item For $1/2\leq t\leq 3/4$ the concordance of the fundamental link pattern is executed in a straightforward way on the first $k-1$ ``circle'' components, and with a modification accounting for additional trivial arcs on the graph-component.
        \item Finally, for $3/4\leq t\leq1$ the resulting $k-1$ disjoint circles are capped off with 2-disks such that only the planar embedding of $G$ remains. We have shown that $G$ is slice.
    \end{enumerate}
    To finish the proof we need to describe the band attachment process and the modification to the link concordance. 

    Given a fundamental link pattern, we get a fundamental cycle basis of $G$ and hence a maximal tree $T\subset G$ (the intersections of all those cycles). Then, the goal is to attach bands to edges not in $T$. The bands are to be attached in small neighborhoods of the vertices and should closely follow the maximal tree.  The process is illustrated with an example in Figure~\ref{fig:link_pattern_proof_ex_1}. 
    When $G$ has vertices of higher valence, the bands are attached first to the edges most adjacent to $T$ outwards (Figure~\ref{fig:link_pattern_proof_ex_2}).

    We claim that after attaching $k-1$ bands and some necessary isotopies the graph now looks exactly as described in step (ii) above. This is clear after we realize that attaching bands as described in Figures~\ref{fig:link_pattern_proof_ex_1} and~\ref{fig:link_pattern_proof_ex_2} corresponds exactly to producing a fundamental cycle given a maximal tree.

    Finally, given the graph in the form described above at $t=1/2$, the concordance of the link pattern can be ``applied'': graph concordance follows the link concordance for $L_1, \ldots, L_{k_1}$. For a component of a fundamental link pattern homeomorphic to the graph, the concordance $L_k\times I$ can be extended to a tubular neighborhood $L_k\times D^2\times I$ and hence to a graph $G'$ that lies in a neighborhood of $L_k$ by construction.

    At this point, we have found a concordance from $G$ to some planar spatial graph $P$. As $G$ is framed, we can extend the framing to the whole concordance $G\times I$ using Proposition~\ref{prop:framing_extend} to finish the proof.
\end{proof}

\begin{figure}
    \centering
    \includegraphics[width=0.7\textwidth]{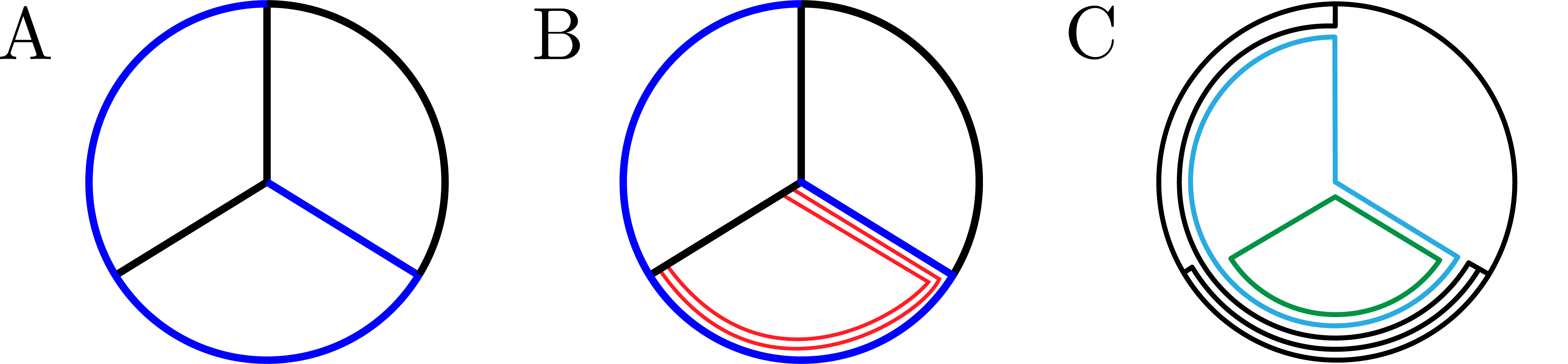}
    \caption{An example process of attaching bands for stage (ii) of the concordance in the proof of Theorem~\ref{thm:fnd_link_pattern}. (A) Initial graph $G$. The surface $\Sigma$ is just a neighborhood of $G$ in the plane, and the maximal tree (determined by the fundamental cycle basis) is shown in purple. (B) The first band (shown in red) is attached. (C) A final picture after all bands are attached.} 
    \label{fig:link_pattern_proof_ex_1}
\end{figure}

\begin{figure}
    \centering
    \includegraphics[width=0.7\textwidth]{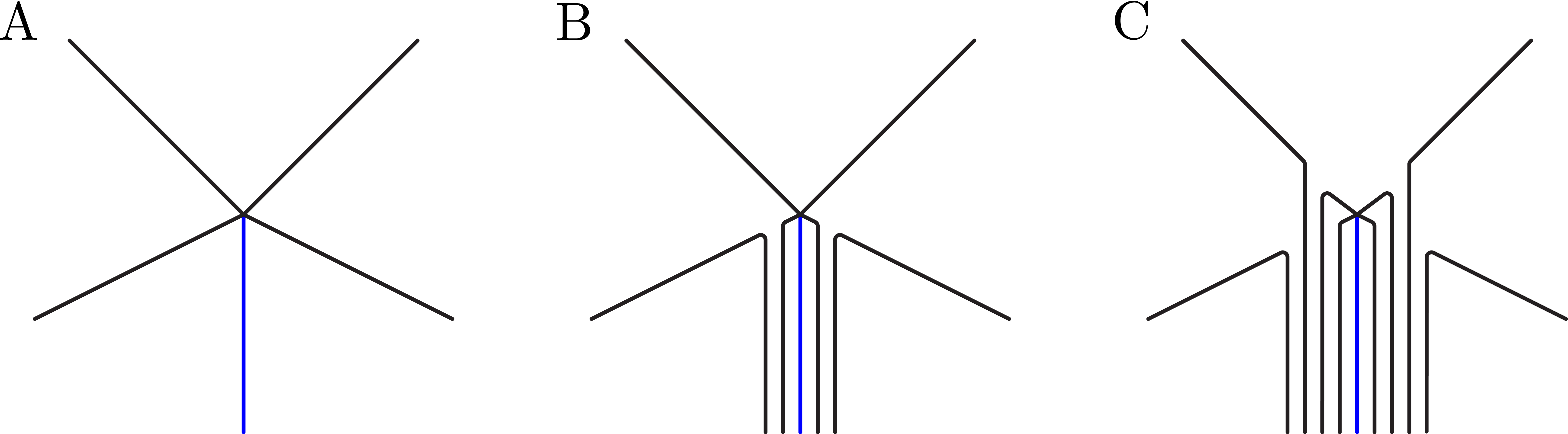}
    \caption{Attaching bands to $G$ around a vertex of valence five.}
    \label{fig:link_pattern_proof_ex_2}
\end{figure}

\section{Main Theorem}

Bringing together the results in Sections~\ref{sec:lk} and \ref{sec:link_patt}, we are able to reduce sliceness of spatial graphs to a statement about framings together with sliceness of links.

{
    \renewcommand{\thetheorem}{\ref{thm:main}}
    \begin{theorem} 
        Let $G$ be a spatial graph with abstract topology $\Gamma$, Let $\{K_i\}$ be its constituent knots constructed from a maximal tree as in Section~\ref{subsec:linking_homology}, and $\Sigma$ an arbitrary framing of $G$. Then, $G$ being slice is equivalent to the combination of the following two conditions:
        \begin{enumerate}[label=(\roman*)]
            \item The images of push-offs of constituent knots are zero in the module of framings of $G$, $\mathbf K(G) = 0 \in \Fr(\Gamma)$;
            \item if (i) is true, then there is a framing $\Sigma_0$ such that each for push-off we have $[K_i^+] = 0 \in H_1(S^3- G)$. The second condition then is: a fundamental link pattern in $\Sigma_0$ is slice.
        \end{enumerate}
    \end{theorem}
    \addtocounter{theorem}{-1}
}

\begin{proof}
    If $G$ is slice, Corollary~\ref{cor:lk_application_slice} gives condition (i). Using Proposition~\ref{prop:framing_extend} to frame the concordance from $G$ to a planar graph $P$ with its blackboard framing, we see that Theorem~\ref{thm:link_pattern} gives condition (ii).

    Conversely, condition (i) on constituent knots allows us apply Proposition~\ref{prop:all_framings} to find a sequence of operations by which any framing $\Sigma$ can be turned into $\Sigma_0$, making all linking numbers identically zero. Condition (ii) gives the hypothesis for Theorem~\ref{thm:fnd_link_pattern}, and that theorem finishes the proof in the opposite direction.
\end{proof}
As the $\theta$-curve in Example~\ref{ex:linking_ex1} shows, for some graph topologies condition (i) is always satisfied, and concordance of a graph is determined by concordance of a certain link.

\begin{example}
    Let $H$ be a spatial graph in Figure~\ref{fig:handcuff}, having an abstract topology $\Phi$ which we call a ``handcuff.'' The homology of the complement of $H$ is $H_1(S^3 - H) \cong \Z^2$ as there are two constituent knots formed by edges $v_1$ and $v_2$. The bridging edge $v_0$ forms a single possible edge cut, but as $v_0$ does not intersect with any of the constituent knots, this edge cut does not affect the linking numbers. If $x_i$ are meridians of $v_i$ in the first copy of $H_1(S^3 - H)$ and $y_i$ are meridians of $v_i$ in the second copy, we have
    \[
        \Fr(\Phi) = \frac{\langle x_1,x_2,y_1,y_2 \mid x_2 - y_1 \rangle}{\langle x_1,y_2\rangle} =  \langle y_1\rangle = \Z.
    \]
    As the two constituent knots are meridians of each other, we have $\mathbf K(H) = (y_1) \neq 0 \in \Fr(\Phi)$, so $H$ is not slice. Indeed, sliceness of $H$ would imply sliceness of the Hopf link, which is false.
\end{example}

\begin{example} Let us consider $\theta$-curves once again. As we have seen in Example~\ref{ex:linking_ex2}, the linking number invariant vanishes for these spatial graphs. Theorem~\ref{thm:fnd_link_pattern} was shown for $\theta$-curves by K. Taniyama in \cite{taniyama1993cobordism}. As an illustration, consider the spatial graph $G$ in Figure~\ref{fig:theta_5_1_example}A which is commonly labeled as $5_1$ \cite{moriuchi2009table}. We see that every constituent knot of $G$ is an unknot, and that there exists a framing $\Sigma_0$ of $G$ making all linking numbers $\lks(K_i,K_j)$ zero (such framing is blackboard except for a negative half-twists on the central vertical edge). Figure~\ref{fig:theta_5_1_example}B shows a fundamental linking pattern $L$ in $\Sigma_0$. The link $L$ has signature $-1$, therefore it is not slice, therefore $G$ is not slice.
\end{example}

\begin{example}\label{ex:theta 3 slice}
As a converse example in which a fundamental link pattern can prove sliceness of a spatial graph, we consider a graph $H$ in Figure~\ref{fig:lpat_slice_example}(A) having an abstract topology of $\Theta_3$. The group $\Fr(\Theta_3)\cong \Z$ is no longer trivial, but luckily the invariant $\mathbf K(H)$ can be computed to be zero using a blackboard framing $\Sigma_0$ of $H$. A fundamental link pattern $L$ in $\Sigma$ is shown in Figure~\ref{fig:lpat_slice_example}B, and it can be seen that $L$ is isotopic to a trivial link. Theorem~\ref{thm:fnd_link_pattern} then implies that $H$ is slice.
\end{example}

\begin{remark}
    Theorem~\ref{thm:main} allows us to import concordance invariants of knots to obstruct sliceness of spatial graphs. For example, if the fundamental link pattern fails the Fox-Milnor condition on the Alexander polynomial, then the spatial graph is not slice. By contrast, there is a version of the Alexander polynomial for spatial graphs which can be defined, for example, through the fundamental group of the graph complement \cite{mellor2018invariants}. The Alexander polynomial for slice spatial graphs does not need to satisfy the Fox-Milnor condition, as the example in Figure~4 of \cite{mellor2018invariants} shows.
\end{remark}

\begin{figure}
    \centering
    \includegraphics[width=0.2\textwidth]{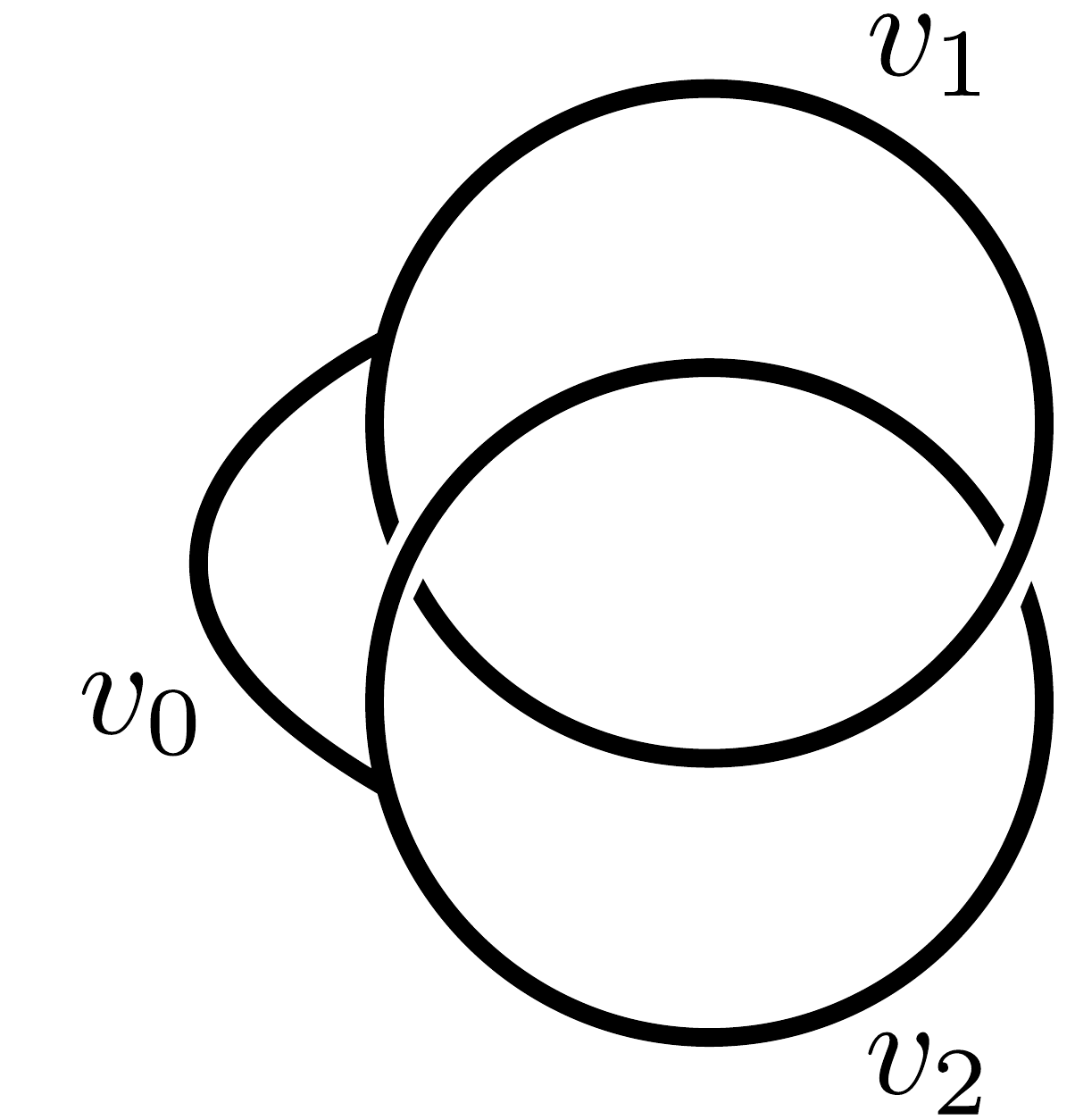}
    \caption{A handcuff spatial graph $H$.}
    \label{fig:handcuff}
\end{figure}

\begin{figure}
    \centering
    \includegraphics[width=0.48\textwidth]{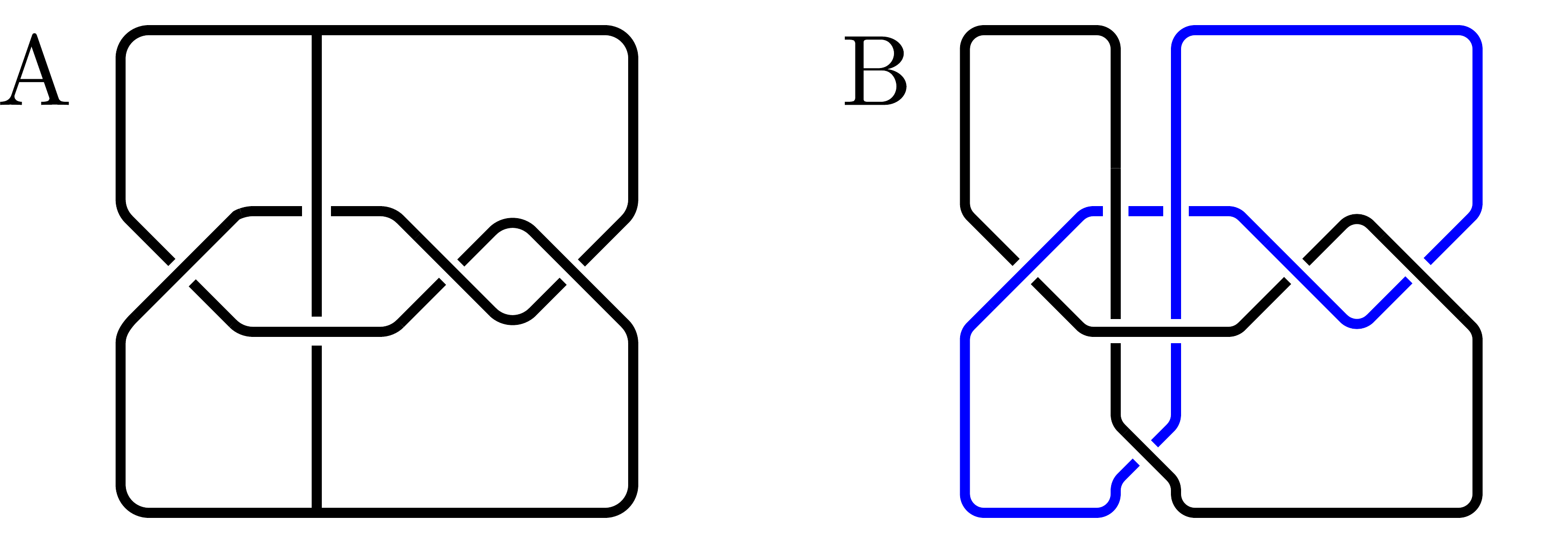}
    \caption{(A) A $\theta$-curve $G$ with five crossings. (B) A fundamental link pattern in the framing $\Sigma_0$ of $G$.}
    \label{fig:theta_5_1_example}
\end{figure}

\begin{figure}
    \centering
    \includegraphics[width=0.65\textwidth]{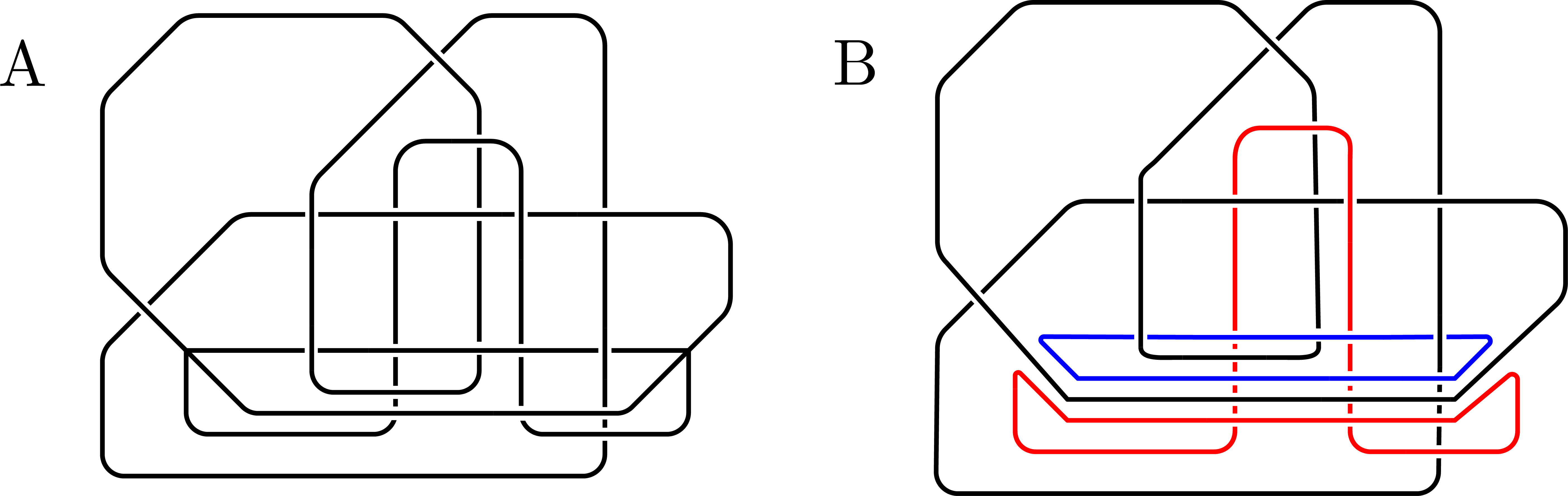}
    \caption{(A) A spatial graph $H$ with abstract topology $\Theta_3$. (B) A fundamental link pattern in the framing $\Sigma_0$ of $H$.}
    \label{fig:lpat_slice_example}
\end{figure}

\bibliographystyle{amsplain}
\bibliography{main.bib}

\end{document}